\documentclass[reqno]{amsart} 
\usepackage{amsfonts}

\usepackage{amssymb}
\usepackage{amsmath}
\usepackage{amsthm}
\usepackage{graphics}
\usepackage{hyperref}

\newcommand{\Oh}{\mathrm{O}}
\newcommand{\oh}{\mathrm{o}}
\def\JELname{\textbf{JEL Classification}\enspace}
\def\JEL#1{\par\addvspace\medskipamount{\rightskip=0pt plus1cm
\def\and{\ifhmode\unskip\nobreak\fi\ $\cdot$
}\noindent\JELname\ignorespaces#1\par}} 

\def\CATname{\textbf{Subject Category}\enspace}
\def\CAT#1{\par\addvspace\medskipamount{\rightskip=0pt plus1cm
\def\and{\ifhmode\unskip\nobreak\fi\ $\cdot$
}\noindent\CATname\ignorespaces#1\par}} 

\newcommand{\5}{\mathchoice{\mskip-1.5mu}{\mskip-1.5mu}{}{}}

\newcommand{\PP}{{\mathbb P}}    
\newcommand{\E}{{\mathbb E}}     

\newcommand{\prob}[1]{\operatorname{\PP}\5\5\left[ #1 \right]}

\newcommand{\poisson}[1]{\operatorname{Poisson}\5\5\left(#1\right)}

\newcommand{\NBin}[2]{\operatorname{NBin}\5\5\left(#1,#2\right)}

\newtheorem{theorem}{Theorem}

\newtheorem{prop}{Proposition}

\newtheorem{corollary}[theorem]{Corollary}

\theoremstyle{definition}
\newtheorem{definition}[theorem]{Definition}

\theoremstyle{remark}

\numberwithin{theorem}{section}
\numberwithin{equation}{section}

\begin{document}

\title[Efficient Recursions by Computer Algebra]{Finding Efficient Recursions for Risk Aggregation by Computer Algebra}

\author{Stefan Gerhold}
\thanks{This work has been done under the financial support by Christian Doppler Laboratory for Portfolio Risk Management (PRisMa Lab) \url{http://www.prismalab.at/}. S.~Gerhold gratefully acknowledges the fruitful collaboration and support by the Bank Austria Creditanstalt (BA-CA) and the Austrian Federal Financing Agency (ÖBFA) through CDG}
\address{Vienna University of Technology, Wiedner Hauptstra\ss{}e 8--10,
A-1040 Vienna, Austria}
\email{sgerhold@fam.tuwien.ac.at}	 

\author{Richard Warnung}
\thanks{R.~Warnung gratefully acknowledges financial support by WWTF MA13}
\address{Vienna University of Technology, Wiedner Hauptstra\ss{}e 8--10,
A-1040 Vienna, Austria}
\email{rwarnung@fam.tuwien.ac.at}	 

\begin{abstract}
  We derive recursions for the probability distribution of random sums by computer algebra.
  Unlike the well-known Panjer-type recursions, they are of finite order and thus allow
  for computation in linear time. This efficiency is bought by the assumption that the
  probability generating function of the claim size be algebraic. The probability generating
  function of the claim number is supposed to be from the rather general class
  of $D$-finite functions.
\end{abstract}



\maketitle
\JEL{C63} 
\CAT{IM22}

\section{Introduction}
Random sums
\[
  L = X_1 + \dots + X_N,
\]
play a prominent role in risk theory. 
We refer to the~$X_i$, which are independent copies of a discrete random variable $X$, as claims, and to~$N$,
which is independent of the $X_i$, as claim number.
A lot of research has been devoted to recursive calculation
of the distribution of~$L$. The classical Panjer recursion and its numerous extensions~\cite{Sundt:1992,Hesselager:1994,Wang/Sobrero} provide
infinite order linear recursions for this problem for various claim number distributions. Recently, Hipp~\cite{Hipp:Phasetype}
has found that finite order recursions can be obtained for phase-type claim size distributions,
with obvious advantages concerning computation time.
The generalized discrete phase-type distributions are the distributions of hitting times in a finite-state discrete-time Markov chain. They serve as a very flexible class of severity distributions with several useful properties~\cite{Neuts:1981,FittingPhasetype}.
Simple examples are the geometric and the negative binomial distribution (with integral parameter $\alpha$).
Another approach by DePril~\cite{DePril:1986} also leads to finite order recursions for some
claim size distributions, for example for piecewise constant or piecewise linear claim size distributions. But the recursion by DePril is not of finite order for the broad class of algebraic probability generating functions which we cover in this article.

We present an alternative method to obtain recursions of finite order satisfied by the distribution $\prob{L=n}$ of~$L$. Our assumption on
the claim number~$N$ is that its probability generating function (pgf) be a $D$-finite function. This broad class of functions is characterized
by linear differential equations with polynomial coefficients. The pgf of the claims is assumed to be algebraic.
Then the theory of $D$-finite functions ensures the existence of a finite order linear recursion
with polynomial coefficients for the distribution of~$L$. The recursion
can be determined by computer algebra, which avoids tedious hand calculations.

Section~\ref{se:d-finite} collects the parts of the theory of $D$-finite functions that we want to use,
together with their algorithmic realization in computer algebra systems.
The latter is applied in Section~\ref{se:examples} to some concrete examples.
Although we focus on computational efficiency, we do care about numerical stability of the proposed recursions.
Therefore, in Section~\ref{se:stability} we apply the stability theory of finite order recursions to our examples.

\section{$D$-finite functions}\label{se:d-finite}

We will assume throughout that the probability generating (pgf) function
\[
  \varphi_N(z) = \E[z^N]
\]
of the claim number $N$ is of the following kind.
\begin{definition}
  Let $f(z)$ be a function that is analytic at zero. Then
  $f(z)$ is called $D$-\emph{finite} if it satisfies a linear differential equation
  \begin{equation}\label{eq:ode}
    Q_0(z) f(z) + Q_1(z) f'(z) + \dots + Q_d(z) f^{(d)}(z) = 0
  \end{equation}
  with polynomial coefficients $Q_0(z),\dots,Q_d(z)$, not all identically zero.
\end{definition}
Most discrete distributions that are used in practice have $D$-finite pgfs. See, e.g., the comprehensive
list of hypergeometric distributions in Johnson, Kemp, Kotz~\cite{JoKeKo05}.
Many properties of $D$-finite functions follow from the classical theory
of ordinary differential equations~\cite{Ince:ODE}; see Stanley~\cite{St80,St99} for an introduction from
a combinatorial viewpoint including a proof of the following result.
\begin{theorem}\label{thm:d-finite}
\quad
\begin{itemize}
  \item[(i)] An analytic function $f(z)=\sum_{n\geq0}a_n z^n$ is $D$-finite if and only if its coefficient sequence $(a_n)_{n\geq0}$
  satisfies a finite order linear recursion with polynomial coefficients.
  \item[(ii)] The sum and the product of two $D$-finite functions are $D$-finite. The composition $f(g(z))$ of a $D$-finite function $f(z)$
  and an algebraic function $g(z)$ is $D$-finite.
\end{itemize}
\end{theorem}
Part~(i) says that a differential equation of the form~\eqref{eq:ode} always translates into a recursion
\[
  R_0(n) a_n + R_1(n) a_{n+1} + \dots + R_e(n) a_{n+e} = 0, \qquad n\geq0,
\]
with polynomial coefficients $R_k(n)$ for the power series coefficients $(a_n)_{n\geq0}$ of $f(z)$, and vice versa.
Therefore, the distributions from the Panjer class are simple examples of distributions with
$D$-finite pgf. More examples of such distributions can be found, e.g., in Panjer and Willmot~\cite{Panjer/Willmot:Difference}.
Hesselager~\cite{Hesselager:1994} has found a recursion for the distribution of $L$ for claim number distributions with arbitrary $D$-finite pgf.
It is valid for any claim size distribution, but is of infinite order in general.
We, on the other hand, aim at finite order recursions for the distribution of $L$.
Part~(ii) of Theorem~\ref{thm:d-finite} is central for our approach.
Recall that an algebraic function $g(z)$ satisfies $P(z,g(z))\equiv0$ for some non-trivial bivariate polynomial $P$.
We explicitly note the result that part~(ii) of Theorem~\ref{thm:d-finite} implies in our situation.

\begin{corollary}\label{cor:composition}
  If the pgf $\varphi_N(z)$ of the claim numbers is $D$-finite and the pgf $\varphi_{X}(z)$
  of the severity distribution is algebraic, then
  \[
    \varphi_L(z) = \varphi_N(\varphi_{X}(z)) = \sum_{n\geq0}a_n z^n
  \]
  is $D$-finite, and its coefficients $a_n$ satisfy a linear
  recurrence of finite order with polynomial coefficients.
\end{corollary}

Examples of admissible severity distributions include those whose pgfs are polynomials or rational functions, in particular,
the phase-type distributions mentioned in the introduction.
Furthermore, consider the negative binomial distribution $\NBin{\alpha}{p}$ with $\alpha>0$ and $p \in (0,1)$, i.e.,
\[
  \PP[X=n] = {\alpha+n-1 \choose n} p^{\alpha}(1-p)^n, \qquad n\geq0.
\]
Its pgf
\[
  \varphi_{X}(z) = \left(\frac{p}{1-(1-p)z}\right)^{\alpha}
\]
is algebraic if $\alpha$ is a rational number.
Other examples of distributions with algebraic pgf are the binomial distribution, the discrete Mittag-Leffler distribution,
and its generalization, the discrete Linnik distribution~\cite{JoKeKo05}. Again,
parameters appearing in the exponent have to be constrained to the rational numbers.
The distribution of the number of games lost by the ruined gambler in the classical gambler's ruin problem~\cite{JoKeKo05}
has an algebraic pgf, too.

All operations described in Theorem~\ref{thm:d-finite} are constructive
and have been implemented in computer algebra packages. We have done the computations below with Mallinger's {\sf Mathematica}
package {\sf GeneratingFunctions}~\cite{Ma96}. Another possible choice would have been Salvy and Zimmermann's {\sf Maple}
package {\sf gfun}~\cite{SaZi94}.
The software requires an algebraic equation for the claim size pgf and a differential equation
for the claim number pgf. We assume that our claim size pgf is an explicit algebraic (or even rational) function,
so the first of these two equations is obvious in our examples. As for the second one, suppose we have a closed form expression for
the $D$-finite claim number pgf.
A differential equation for it can be built up step by step, by starting from obvious differential equations (for
the exponential function, say) and using commands that realize the closure properties
{} from Theorem~\ref{thm:d-finite}.
Finally, the package allows to convert the differential equation that we have thus found for
$\varphi_L(z)$ into the desired finite order recursion for its coefficients.

The package {\sf GeneratingFunctions} can deal with undetermined parameters.
For instance, the differential equation $(a z+b)f'(z) - cz^2 f(z)=0$ with parameters $a,b,c$ would be a valid input,
and we could, e.g., compute a recurrence relation for the power series coefficients of $f(z)$.
When inputting algebraic relations, however, exponents must be fixed: We can
specify $f(z)^{2}=(a z + b)^3$, e.g., but not $f(z)^\alpha=(a z + b)^3$ with undetermined~$\alpha$.

\section{Examples}\label{se:examples}

We illustrate our approach by three examples. The intermediate differential equations obtained during the process are not displayed.
First we consider $N \sim \NBin{\alpha}{p}$ and $X \sim \NBin{\beta}{q}$ with $p,q \in (0,1)$. Then the pgf of the aggregate loss is given by
\begin{equation}\label{eq:pgf negbin negbin}
		\varphi_L(z) = \left(\frac{p}{1-(1-p)(\frac{q}{1-(1-q)z})^\beta}\right)^\alpha.
\end{equation}
Our goal is to find a differential equation for $\varphi_L(z)$ and thence a recurrence relation for $a_n=\prob{L=n}$.
As a byproduct of the stability analysis in Section~\ref{se:stability}, we will obtain the asymptotics of $a_n$
for general $\alpha$ and $\beta$. To compute a recurrence for $a_n$ by computer algebra, however,
these parameters have to be concrete rational numbers (see the last paragraph of Section~\ref{se:d-finite}). We choose $\alpha=\tfrac12$ and $\beta=\tfrac13$.
Then the function $\varphi_N(z)$ satisfies the algebraic equation
\[
  (1-(1-p)z)\varphi_N(z)^2 = p,
\]
{}from which the command {\sf AlgebraicEquationToDifferentialEquation} computes a differential equation for $\varphi_N(z)$.
{}From the latter and the algebraic equation
\[
  (1-(1-q)z)\varphi_X(z)^3 = q
\]
of $\varphi_X(z)$, the command {\sf AlgebraicCompose} (cf.\ Corollary~\ref{cor:composition}) derives the differential equation
\begin{align}\label{eq: ODE}
	\begin{split}
	&-5(p-1)^3(q-1)^3 q f(z) +  \\
	&10(q-1)^2(1-z-qz)(16-q+3pq-3p^2q+p^3q-16z+16qz) f'(z) +  \\
	&36(q-1)(1-z+qz)^2(8-3q+9pq-9p^2q+3p^3q-8z+8qz) f''(z) +  \\
	&72(1-z+qz)^3(1-q3pq-3p^2q+p^3q-z+qz) f'''(z) = 0
	\end{split}
\end{align}
for $\varphi_L(z)= \varphi_N(\varphi_X(z))$. Finally, this equation is transformed into the recurrence
\begin{align}\label{eq:rec negbin negbin}
	\begin{split}
	&8n(1+3n)(2+3n)(q-1)^4 a_n	 +  \\
	&(q-1)^3(320 + 896n+864n^2 + 288n^3 - 5q - 46nq - 108n^2q -72 n^3q + 15 pq + 138 n p q + 324 n^2 pq +\\
	&216 n^3 pq -15 p^2 q -138 np^2q -324 n^2p^2q-216 n^3 p^2q+5p^3q+46np^3q + 108 n^2p^3q +72n^3p^3q)	a_{n+1} +  \\
	&(4+2n)(q-1)^2(512 + 648 n + 216 n^2 -113 q - 216 n^2 -113 q - 216 nq -108n^2q + \\
	& 339 pq+648 npq + 324 n^2 pq -339 p^2q - 648 np^2q - 324 n^2p^2q + 113p^3q + 216 n p^3q + 108 n^2 p^3 q) a_{n+2} +  \\
	&36(2 + n)(3 + n)(q-1)(16 + 8 n - 9 q - 6 n q + 27 p q + 18 n p q - 27 p^2 q - 18 n p^2 q + 9 p^3 q + 6 n p^3 q) a_{n+3}\\
   &= 72(2 + n)(3 + n)(4 + n)(-1 + q - 3 p q + 3 p^2 q - p^3 q) a_{n+4}
	\end{split}
\end{align}
for the $a_n$ by {\sf DifferentialEquationToRecurrenceEquation}. Using this recurrence, the probability $a_n$ can be computed with $\Oh(n)$ operations.
We will show in Section~\ref{se:stability} that the computation is numerically stable.

In our second example we suppose that $N \sim \poisson{\Lambda}$ and $\Lambda\sim\mathrm{GIG}(\psi,\chi,\theta)$, the generalized inverse Gaussian distribution with parameters $\theta$ and $\psi,\chi>0$.
In this case we have
\begin{equation}\label{eq:phi_N}
  \varphi_N(z) = \frac{\psi^{\theta/2}}{K_\theta(\sqrt{\psi\chi})} \cdot (\psi+2-2z)^{-\theta/2} \cdot K_\theta(\sqrt{\chi(\psi+2-2z)}),
\end{equation}
where $K_\theta(z)$ is a modified Bessel function of the second kind~\cite[p.~374]{Abramowitz/Stegun}.
Provided that $\theta$ is a rational number, the second factor is an algebraic function, hence $D$-finite.
(Below we will fix $\theta=\tfrac23$.)
The Bessel function $K_\theta(z)$ is $D$-finite for any $\theta$, by virtue of its classical second order differential equation.
Therefore, by Theorem~\ref{thm:d-finite}, our $\varphi_N(z)$ is indeed a $D$-finite function.

As for the severities, we take them to be shifted geometrically distributed: $X\sim\mathrm{Geo}(1,q)$ with $q \in (0,1)$, so that
\[
   \varphi_X(z) = \frac{qz}{1-(1-q)z}.
\]
Once again we want to find a differential equation, and thence a recurrence for the power series coefficients, for the function
\begin{equation}\label{eq:phi_N 2}
  \varphi_L(z) = \varphi_N(\varphi_{X}(z)) =: \mathrm{const} \cdot f(\varphi_{X}(z)) \cdot K_\theta(g(\varphi_{X}(z))),
\end{equation}
where $f(z)$ and $g(z)$ are algebraic functions defined according to~\eqref{eq:phi_N}.
The command {\sf AlgebraicEquationToDifferentialEquation} computes a differential equation
for $f(\varphi_{X}(z))$ from the algebraic equation
\[
  f(\varphi_{X}(z))^{-2/\theta} = \psi + 2 -  \frac{2qz}{1-(1-q)z}.
\]
As mentioned above, this works for any rational $\theta$; to perform the calculation step, we have to fix its value
though, say $\theta=\tfrac23$.
A differential equation for $K_{2/3}(g(\varphi_{X}(z)))$ can be found with {\sf AlgebraicCompose}. It takes as input
the differential equation of $K_{2/3}(z)$ and the algebraic equation
\[
  g(\varphi_{X}(z))^2 = \chi(\psi + 2 - \frac{2qz}{1-(1-q)z}).
\]
Now that we have a differential equation for each of the two (non-constant) factors in~\eqref{eq:phi_N 2}, the command {\sf DECauchy}
computes a differential equation for $\varphi_N(\varphi_{X}(z))$, which is transformed into a recursion
for its power series coefficients~$a_n$ by {\sf DifferentialEquationToRecurrenceEquation}. The recurrence we find is
\begin{multline}\label{P-GIG rec}
  3n(1 + n)(-1 + q)^3(-2 - \psi + \psi q)\cdot a_n \\
  + 2(1 + n)(-1 + q)^2(-18 - 12 n - 
  9 \psi - 6 n\psi + q + 3 n q + 9 \psi q + 6 n \psi q)\cdot a_{n+1} \\
  + (144 + 144 n + 36 n^2 + 72 \psi + 72n \psi + 
  18 n^2 \psi - 188 q - 202 n q - 54 n^2 q - 144 \psi q - 144 n \psi q \\
  - 36 n^2 \psi q + 44 q^2 - 3 \chi q^2 +   58 n q^2 + 18 n^2 q^2 + 72 \psi q^2 + 72 n \psi q^2 + 
  18 n^2 \psi q^2)\cdot a_{n+2} \\
  +  2 (3 + n) (-30 - 12 n - 15 \psi - 6 n \psi + 19 q + 
  9 n q + 15 \psi q + 6 n \psi q)\cdot a_{n+3} \\
  +  3 (3 + n) (4 + n) (2 + \psi)\cdot a_{n+4}= 0.
%
\end{multline}
Summing up, if $N$ has a mixed Poisson distribution $\poisson{\Lambda}$ with $\Lambda\sim \mathrm{GIG}(\psi,\chi,\tfrac23)$,
and $X\sim\mathrm{Geo}(1,q)$, then we can compute the total loss probabilities with $\Oh(n)$ operations
by the recurrence~\eqref{P-GIG rec}.

The third example we consider is
$N \sim \poisson{\lambda}$ and $X \sim \NBin{\frac12}{p}$ with $p \in (0,1)$. To determine a differential equation
for the pgf
\[
	\varphi_L(z) = \exp\left(\lambda\left(\Bigl(\frac{p}{1-(1-p)z}\Bigr)^{1/2} -1\right)\right),
\]
we use again the command {\sf AlgebraicCompose}. Its input are the algebraic equation
\[
  (1-(1-p)z)\varphi_X(z)^2 = p
\]
and the differential equation
\[
  \varphi_N'(z) = \lambda \varphi_N(z).
\]
{\sf AlgebraicCompose} then finds the differential equation
\begin{equation}\label{eq:ode poisson negbin}
  \lambda^2p(1-p)^2 \varphi_L(z) + 6(1-p)(1-(1-p)z)^2 \varphi_L'(z) - 4(1-(1-p)z)^3 \varphi_L''(z) = 0.
\end{equation}
Using {\sf DifferentialEquationToRecurrenceEquation}, we obtain the recurrence
\begin{multline}\label{eq:rec poisson negbin}
 2n(2n+1)(1-p)^3 a_n - (1-p)^2(-\lambda^2p+12n^2+24n+12) a_{n+1} \\
 + 6(n+2)(2n+3)(1-p) a_{n+2} = 4(n+2)(n+3)a_{n+3}
\end{multline}
for the probabilities $a_n$.

\section{Numerical Stability and Asymptotics}\label{se:stability}

The computation of a sequence by a linear recurrence relation is numerically stable if the sequence grows
at least as fast as any other solution of the recurrence~\cite{Wimp:1984,Panjer/Wang}.
This is intuitively clear, since rounding errors will always add a portion of each member
of a fundamental system of the recurrence to the solution we are computing. Asymptotically dominant solutions will therefore wipe out
subordinate solutions in the long run.
In this section we show how to apply methods from asymptotic analysis to the examples from Section~\ref{se:examples}.
The growth of the coefficients $a_n$ depends on the location and nature of the singularity of the generating function
$\varphi_L(z)$ that is closest to the origin~\cite{FlSe07}.
To assess the growth of the other solutions of our recurrences, we have to analyze
the dominating singularity of the differential equation for $\varphi_L(z)$.
If it is regular, then a fundamental system can in principle be determined by Frobenius' method. Flajolet and Odlyzko's
singularity analysis~\cite{FlOd90,FlSe07} allows to obtain the growth rate of the power series coefficients of these solutions.
Then, hopefully, we can read off that the solution we are interested in dominates the other ones.
This fairly general method works in the first two examples from Section~\ref{se:examples}.

\begin{prop}\label{prop: NegBinNegBin1}
  Let $N \sim \NBin{\alpha}{p}$ and $X \sim \NBin{\beta}{q}$ with $p,q \in (0,1)$.
  Then the probabilities $a_n=\prob{L=n}$ satisfy
\begin{align}\label{negbin negbin asmpt}
	a_n \sim C z_1^{-n} n^{\alpha-1}
\end{align}
  as $n\to\infty$, where
  \[
    z_1 = \frac{1-q(1-p)^{1/\beta}}{1-q} \quad \text{and} \quad C=\frac{(pq)^{\alpha}(1-p)^{\alpha/\beta}}{\Gamma(\alpha) \beta^\alpha(1-q(1-p)^{1/\beta})^\alpha}.
  \]
\end{prop}

\begin{proof}
The dominating singularity of $\varphi_L(z)$ is located at $z=z_1$.
Moving the singularity to $z=1$ and putting $c:=1-q(1-p)^{1/\beta}$, we find
\[
	\varphi_L(zz_1) = \left(\frac{p}{1-(1-p)\left(\frac{q}{1-c z}\right)^\beta}\right)^\alpha.
\]
{}From the expansion
\begin{align*}
  \left(\frac{q}{1-c z}\right)^\beta &= \left(\frac{q}{1-c}\right)^\beta \left(1 + \frac{c}{1-c}(1-z) \right)^{-\beta} \\
  &= \frac{1}{1-p}\left(1 - \frac{\beta c}{1-c}(1-z) + \Oh((1-z)^2) \right)
\end{align*}
we thus obtain
\begin{align}
  \varphi_L(zz_1) &= p^\alpha \left( \frac{\beta c}{1-c}(1-z) + \Oh((1-z)^2) \right)^{-\alpha} \notag \\
  &= \frac{(pq)^{\alpha}(1-p)^{\alpha/\beta}}{\beta^\alpha(1-q(1-p)^{1/\beta})^\alpha} (1-z)^{-\alpha}  + \Oh((1-z)^{-\alpha+1}) \label{eq:phi asympt}
\end{align}
as $z$ tends to $1$. The $n$-th power series coefficient of $(1-z)^{-\alpha}$
asymptotically equals $n^{\alpha-1}/\Gamma(\alpha)$~\cite[Chapter VI]{FlSe07}. The coefficients of the error term in~\eqref{eq:phi asympt}
are of smaller order~\cite[Chapter VI]{FlSe07}, whence the desired result.
We note that an asymptotic expansion to arbitrary order can be obtained in the same way.
\end{proof}

To be able to derive a concrete recursion we had to assign values to the parameters $\alpha$
and $\beta$ in Section~\ref{se:examples}.

\begin{prop}\label{prop: NegBinNegBinStab}
  Let $N \sim \NBin{\tfrac12}{p}$ and $X \sim \NBin{\tfrac13}{q}$ with $p,q \in (0,1)$.
  Then the computation of the $a_n$ by the recursion~\eqref{eq:rec negbin negbin} is numerically stable.
\end{prop}
\begin{proof}
We show, by Frobenius' theory of power series solutions of differential equations~\cite{Ince:ODE}, that no solution of~\eqref{eq:rec negbin negbin}
grows faster than~\eqref{negbin negbin asmpt} (with $\alpha=\tfrac12$ and $\beta=\tfrac13$). Instead of working directly with~\eqref{eq: ODE}, we transform~\eqref{eq:rec negbin negbin} into a differential equation. This is necessary
because new solutions might creep in when passing from a differential equation to a recurrence or vice versa. The new differential equation, called
$\mathcal{E}$ in what follows, is solved by
all generating functions of solutions of~\eqref{eq:rec negbin negbin}. It is of order four and has the same
leading coefficient as~\eqref{eq: ODE}. By equating this coefficient to zero, we find the dominating singularity $z_1= (1-q(1-p)^{3})/(1-q)$.
The indicial polynomial of $\mathcal{E}$ is found by plugging in a generalized power series $(z-z_1)^w \sum_{k=0}^\infty A_k (z-z_1)^k$
with undetermined $w$ and $A_k$ and equating the coefficient of the lowest power of $z_1-z$ to zero.
The lowest power turns out to be $(z-z_1)^{w-3}$ with the coefficient $(w-2)(w-1)w(1+2w)$ times a constant.
The degree of this indicial polynomial equals the order of the differential equation, hence $z_1$ is a regular singularity of $\mathcal{E}$.
The roots of the indicial polynomial are the possible values of the exponent $w$ in the generalized power series solution.
The root $w=-\tfrac12$ leads, by singularity analysis, to a solution of $\mathcal{E}$ whose coefficients
grow like~\eqref{negbin negbin asmpt} (with $\alpha=\tfrac12$ and $\beta=\tfrac13$). The other solutions of the fundamental system that Frobenius' method
yields have either no singularity at $z_1$ or a logarithmic singularity at $z_1$. The latter type occurs
because some roots of the indicial polynomial differ by integers, and the
coefficients of the corresponding solutions of $\mathcal{E}$ grow like $1/n$ times a power of $\log n$~\cite[Chapter VI]{FlSe07}.
\end{proof}

The second example from Section~\ref{se:examples} can be treated analogously:

\begin{prop}\label{prop: PoissonInvGaussian}
  Let $N \sim \poisson{\Lambda}$ and $\Lambda\sim\mathrm{GIG}(\psi,\chi,\theta)$, where $\psi,\chi$ and $\theta$ are positive. Furthermore assume that $X\sim\mathrm{Geo}(1,q)$ with $q \in (0,1)$. 
  Then the probabilities $a_n=\prob{L=n}$ satisfy
  \begin{equation}\label{eq: ping asmpt}
    a_n \sim C \chi^{-\theta/2} D^{-\theta} (2n)^{\theta-1} z_1^{-n} 
  \end{equation}
as $n\to\infty$, where
  \[
    C=\frac{\psi^{\theta/2}}{K_\theta(\sqrt{\chi \psi})}, \quad z_1=\frac{1}{1- \psi q/(2+\psi)}, \quad \text{and} \quad  D=\frac{(2+\psi)(2+\psi(1-q))}{2q}.
  \]
\end{prop}
\begin{proof}
We proceed analogously to Proposition~\ref{prop: NegBinNegBin1}.
The dominating singularity of $\varphi_L(z)$ is located at $z=z_1$.
We use the expansion
\[
  K_\theta(z) = 2^{\theta-1} \Gamma(\theta) z^{-\theta} + \Oh(z^{\min\{\theta,2-\theta\}}),  \qquad z \to 0,
\]
valid for $\theta>0$.
{}From this we find
\[
  \varphi_L(z_1 z) \sim C D^{-\theta} \chi^{-\theta/2} 2^{\theta-1}\Gamma(\theta) (1-z)^{-\theta}, \qquad z \to 1.
\]
The result now follows from singularity analysis, since
the coefficients of $(1-z)^{-\theta}$ asymptotically equal $n^{\theta-1}/\Gamma(\theta)$.
Once again, an asymptotic expansion to arbitrary order can be readily obtained.
\end{proof}

\begin{prop}
  Let $N \sim \poisson{\Lambda}$ and $\Lambda\sim\mathrm{GIG}(\psi,\chi,\frac23)$. Furthermore assume that $X\sim\mathrm{Geo}(1,q)$ with $q \in (0,1)$. 
  Then the computation of the $a_n$ by the recursion~\eqref{P-GIG rec} is numerically stable.
\end{prop}

\begin{proof} Completely analogous to the proof of Proposition~\ref{prop: NegBinNegBinStab}.
The indicial polynomial is $w(w-1)(w-2)(3w+2)$. The root $w = -\frac23$ leads to a solution whose coefficients grow like~\eqref{eq: ping asmpt} (with $\theta=\frac23$), whereas the coefficients of the other solutions grow slower.
\end{proof}

The approach we have just illustrated works whenever the dominating singularities of the differential equation
satisfied by $\varphi_L(z)$ are regular. In our third example, however, the dominating singularity is irregular. In general, it is difficult to say anything about
the growth order of the power series coefficients of the solutions in this case.
In our example, though, all solutions of the differential equation for $\varphi_L(z)$ can be expressed in closed form,
and their coefficients can be analyzed by Cauchy's integral formula and the saddle point method. 

\begin{prop}
  Let $N \sim \poisson{\lambda}$ and $X \sim \NBin{\frac12}{p}$ with $p \in (0,1)$.
  Then the computation of the $a_n$ by the recursion~\eqref{eq:rec poisson negbin} is numerically stable, and the probabilities satisfy
  \[
    a_n \sim \frac{\lambda^{1/3}p^{1/6}}{2^{1/3}\sqrt{3\pi}} (1-p)^n n^{-5/6 }\exp\left(3p^{1/3}(\lambda/2)^{2/3} n^{1/3} - \lambda\right)
  \]
  as $n\to\infty$.
\end{prop}
\begin{proof}
  We present the proof for $p=\tfrac12$ and $\lambda=1$. The general case yields no additional complications.
  The functions
  \[
    \left\{ \exp \frac{\pm 1}{\sqrt{2-z}} \right\}
  \]
  form a fundamental system for the differential equation~\eqref{eq:ode poisson negbin}.
  Therefore, a fundamental system of the third order recursion~\eqref{eq:rec poisson negbin} is given by our $a_n$, the coefficients
  of $\exp(-1/\sqrt{2-z})$, and $(1,0,0,0,\dots)$.
  We have to show that the coefficients $a_n$ of $\varphi_L(z)=\exp(1/\sqrt{2-z})$ have the announced asymptotic behavior, and that those
  of $\exp(-1/\sqrt{2-z})$ grow slower. (The additional solution $(1,0,0,0,\dots)$ of~\eqref{eq:rec poisson negbin} cannot make the computation
  unstable, of course.) To do so, we appeal to Cauchy's integral formula:
  \begin{align*}
    a_n &= \frac{1}{2\mathrm{i}\pi} \int_{|z|=r} \frac{\varphi_L(z)}{z^{n+1}}\mathrm{d}z \\
     &= \frac{1}{2\pi r^n} \int_{-\pi}^\pi \mathrm{e}^{-\mathrm{i}n\theta} \varphi_L(r \mathrm{e}^{\mathrm{i}\theta}) \mathrm{d} \theta,
       \qquad 0 < r < 2.
  \end{align*}
  We will determine the asymptotics of the integral by the saddle point method~\cite{deBr58,FlSe07}.
  To find an approximate saddle point, we equate the derivative
  of the integrand to zero, which leads to the equation
  \[
    4n^2(2-z)^3 = z^2.
  \]
  Clearly, we must have $z\to2$ as $n\to\infty$ here. By plugging $z=2-u$ with unknown $u = \oh(1)$ into the equation, we obtain
  $u\sim n^{-2/3}$. Therefore, we choose the new integration contour $|z|=r:=2-n^{-2/3}$. The dominant part of
  the integral arises near the saddle point, for $\theta=\Oh(n^{-\alpha})$, where $\alpha$ is a fixed parameter with $\tfrac79<\alpha<\tfrac56$.
  Outside this central part we have
  \[
    \cos \theta \leq \cos(n^{-\alpha}) = 1 - \tfrac12 n^{-2\alpha} + \Oh(n^{-4\alpha}),
  \]
  and using this estimate in
  \begin{equation*}
    |2-z|^{-1/2} = (4-4r\cos\theta +r2)^{-1/4}, 
  \end{equation*}
  we obtain
  \begin{equation}\label{eq:tail}
    \exp \frac{1}{|2-z|^{1/2}} \leq \exp(n^{1/3}-n^{5/3-2\alpha}) (1+\oh(1)).
  \end{equation}
  To calculate the central part of the integral, we compute the second order approximation
  \[
    (2-r \mathrm{e}^{\mathrm{i}\theta})^{-1/2} = n^{1/3} + \mathrm{i}n\theta -\tfrac32 n^{5/3}\theta2 + \Oh(n^{7/3-3\alpha}).
  \]
  Since
  \[
    \int_{-n^{-\alpha}}^{n^{-\alpha}} \exp(-\tfrac32 n^{5/3}\theta2) \mathrm{d} \theta \sim \sqrt{\frac{2\pi}{3}} n^{-5/6}
  \]
  and $r^{-n} \sim 2^{-n} \exp(\tfrac12 n^{1/3})$,
  we find
  \[
    a_n \sim \frac{1}{2\pi r^n} \int_{-n^{-\alpha}}^{n^{-\alpha}} \mathrm{e}^{-\mathrm{i}n\theta} \exp(\frac{1}{\sqrt{2-r \mathrm{e}^{\mathrm{i}\theta}}}) \mathrm{d} \theta
    \sim \frac{1}{\sqrt{6\pi}} \frac{\exp(\tfrac32 n^{1/3})}{2^n n^{5/6}}.
  \]
  Note that we have shown above that the remaining portion of the integral, where $n^{-\alpha}<|\theta|<\pi$, grows slower, by virtue of
  the factor $\exp(-n^{5/3-2\alpha})$ in~\eqref{eq:tail}.
  
  Now that we have established the asymptotics of $a_n$,
  it remains to show that the coefficients, $b_n$ say, of $\exp(-1/\sqrt{2-z})$ grow slower. To see this,
  we use Cauchy's integral formula with the same contour as above:
  \[
    b_n = \frac{1}{2\pi r^n} \int_{-\pi}^\pi \mathrm{e}^{-\mathrm{i}n\theta}
      \exp(\frac{-1}{\sqrt{2-r \mathrm{e}^{\mathrm{i}\theta}}}) \mathrm{d} \theta.
  \]
  Here the integrand has no saddle point near $z=r$, but a bound good enough for our purpose can still be deduced.
  The tail $|\theta|>n^{-\alpha}$ satisfies the same estimate as for $\varphi_L(z)$. Near the real axis, for $\theta=\oh(n^{-\alpha})$, we have
  \begin{align*}
    |\exp(-(2 - r \mathrm{e}^{\mathrm{i}\theta})^{-1/2})| &\sim \exp(-n^{1/3}+\tfrac32 n^{5/3}\theta^2) \\
     &\leq \exp(-\tfrac12 n^{1/3})
  \end{align*}
  for large $n$,
  which shows that the integral over the central part grows slower than that for $a_n$ (it even tends to zero), whence $b_n=\oh(a_n)$.
\end{proof}

%
%


\addcontentsline{toc}{section}{References}
\bibliographystyle{abbrv}
\bibliography{LiteratureRecursions}
\end{document}